\documentclass[10pt]{amsart}

\usepackage[pagebackref]{hyperref}
\hypersetup{
    colorlinks = true,
    citecolor = [rgb]{0, 0.5, 0.00},
}

\renewcommand*{\backref}[1]{}
\renewcommand*{\backrefalt}[4]{%
    \ifcase #1 (Not cited.)%
    \or        (Cited on page~#2.)%
    \else      (Cited on pages~#2.)%
    \fi}

\usepackage{mathtools}

\usepackage[margin=1.175in]{geometry}
\usepackage{todonotes}
\usepackage{multirow}
\usepackage{longtable}

\usepackage{amsmath}
\usepackage{amsfonts}
\usepackage{amssymb}
\usepackage{float}
\usepackage{amsthm}
\usepackage{comment}
\usepackage{amscd}
\usepackage{amsxtra}
\usepackage{epsfig}

\usepackage{listings}
\usepackage{color}

\definecolor{dkgreen}{rgb}{0,0.6,0}
\definecolor{gray}{rgb}{0.5,0.5,0.5}
\definecolor{mauve}{rgb}{0.58,0,0.82}

\usepackage{color}
\usepackage[all]{xy}
\usepackage{verbatim}

\usepackage{eucal}
\usepackage{mathrsfs}

\usepackage{graphicx}
\usepackage{tikz-cd}

\usepackage{url}
\usepackage{verbatim}

\usepackage{xy}
\xyoption{all}

\usepackage{amsthm}

\usepackage{caption} 
\makeatletter
\def\@tocline#1#2#3#4#5#6#7{\relax
  \ifnum #1>\c@tocdepth 
  \else
    \par \addpenalty\@secpenalty\addvspace{#2}%
    \begingroup \hyphenpenalty\@M
    \@ifempty{#4}{%
      \@tempdima\csname r@tocindent\number#1\endcsname\relax
    }{%
      \@tempdima#4\relax
    }%
    \parindent\z@ \leftskip#3\relax \advance\leftskip\@tempdima\relax
    \rightskip\@pnumwidth plus4em \parfillskip-\@pnumwidth
    #5\leavevmode\hskip-\@tempdima
      \ifcase #1
       \or\or \hskip 1em \or \hskip 2em \else \hskip 3em \fi%
      #6\nobreak\relax
    \hfill\hbox to\@pnumwidth{\@tocpagenum{#7}}\par
    \nobreak
    \endgroup
  \fi}
\makeatother

\newtheorem{lemma}{Lemma}[section]

\newtheorem{theorem}[lemma]{Theorem}

\theoremstyle{definition}

\newtheorem{definition}[lemma]{Definition}

\newtheorem{remark}[lemma]{Remark}

\newcommand{\mi}{\mathrm{i}}
\newcommand{\me}{\mathrm{e}}

\title{Orbit Braid Action on a Finite Generated Group}
\author{Haochen Qiu}

\begin{document}

\maketitle

\begin{abstract}

This paper aims to generalize Artin's ideas in \cite{Artin1925} to establish an one-to-one correspondence between the orbit braid group $B^{orb}_n(\mathbb{C},\mathbb{Z}_p)$ and a quotient of a group formed by some particular $\mathbb{Z}_p$-homeomorphisms of a punctured plane. First, we find a faithful representation of $B^{orb}_n(\mathbb{C},\mathbb{Z}_p)$ in a finite generated group whose generators are corresponding to generators of fundamental group of the punctured plane, by demonstrating that the representation from $B^{orb}_n(\mathbb{C}^{\times},\mathbb{Z}_p)$ to the fundamental group is faithful. Next we investigate some characterizations of orbit braid representation to get our conclusion.
\end{abstract}
\section{Introduction}
Braid groups are fundamental objects in mathematics, which were first defined rigorously and studied by Artin in \cite{Artin1925}. The most typical braid group is $B_n= \pi_1(F(\mathbb{C},n)/\Sigma_n)$, where 
\[F(\mathbb{C},n) = \{(x_1,\dots, x_n) \in \mathbb{C}^{\times n} | x_i \neq x_j \text{ for } i \neq j\} 
\]
 and $\Sigma_n$  is the free action of the symmetric group on $ F(\mathbb{C},n) $, defined by $\sigma(x_1,\dots, x_n) = ( x_{\sigma(1)}, \dots , x_{\sigma(n)})$.
An element of $B_n$ is an isotopy equivalent class of realized braids in $3-$dimension space. By identifying the initial points and the end points of these braids, we could get every knot and link in $S^3$, according to Alexander Theorem \cite{Alexander1923}.

Artin also found an algorithm to compute link group, which is the fundamental group of link's complement in $S^3$, and has been regarded as one of the most important invariants of links in $S^3$. Artin's method is as followed (the third step was simplified by Birman in \cite{BIRMAN1974}):
\begin{itemize}
\item Established a faithful representation of $B_n$ in the fundamental group of $\mathbf{D}^2\setminus Q_n$(the standard disk in $\mathbb{C}$ punctured at $n$ distinct points).
\item Investigate some characteristics of homeomorphisms of $\mathbf{D}^2$  that fix boundary, to identify each such kind of homeomorphism with a braid element.
\item Using $\beta :\mathbf{D}^2\setminus Q_n \rightarrow \mathbf{D}^2\setminus Q_n$ to represent $\beta \in B_n$ through the identification, we put the quotient space
$\{(\mathbf{D}^2\setminus Q_n ) \times I\}/\sim $ with $(z,0)\sim(\beta(z),1)$ into $S^3$.
Use the fibration $(\mathbf{D}^2\setminus Q_n ) \times I \rightarrow S^1 $ to find out that 
the action of $\beta$ on the punctured disk is exactly the presentation of link group of the closure of $\bar{\beta}$.
\end{itemize}
~\\
The theory of orbit braids was upbuilt by Hao Li, Zhi L{\"u} and Fenglin Li in \cite{Li2019}:
\begin{definition}
for a connected topological manifold $M$ of dimension greater than one, which admits an effective $G$-action where $G$ is a finite group, 
\[
F_G(M,n)=\{(x_1,\dots,x_n) \in M^{\times n} | G(x_i) \cap  G(x_j) = \emptyset
\text{ for } i \neq j\}.
\]
Then $\alpha:I \rightarrow F_G(M,n)$ with $\alpha(0) = \mathbf{x} = (x_1,\dots , x_n)$ and $\alpha(1) = g\mathbf{x}_{\sigma}  = (g_1x_{\sigma(1)}, \dots, g_nx_{\sigma(n)})$ for some $(g,\sigma) \in G^{\times n}\times \Sigma_n$ is called an orbit braid in $M\times I$.
\end{definition}
 They have defined an equivalence relation among all orbit braids at an orbit base point, so that all equivalence classes can form a group $\mathcal{B}^{orb}_n(M,G)$:
\begin{definition}
We say that two paths $\alpha,\beta: I \rightarrow F_G(M,n) $ with the same initial points are isotopic with respect to the $G$-action relative to endpoints in $M\times I$ if $\alpha \simeq \beta$ rel $\partial I$.
\end{definition}
Intuitively, a string of an orbit braid could pass through its own orbits but could not pass through other strings and their orbits. 
The geometric presentation of classical braid group $B_n (\mathbb{R}^2 )$ in $\mathbb{R}^2\times I$ gives us much more insights to the case of orbit braid group. 
Thus we begin our work from the case of $\mathbb{C} \approx \mathbb{R}^2$ and $\mathbb{C}^{\times} = \mathbb{C} \setminus \{0\}$ with the typical action $\mathbb{Z} \curvearrowright \mathbb{C}$ defined by $(n,z) \mapsto \me^{\frac{2n\pi i}{p}}z$.

According to \cite{Li2019}, the generators of $B^{orb}_n(\mathbb{C},\mathbb{Z}_p)$ are $b$ and $b_k(k = 0, \dots , n-2)$ where $b_k$ could be realized as a representative path
\[
\alpha_k(t) =  (1 + \mi,\dots, (k+1) + (k+2)\mi+\me^{-\frac{\pi}{2}\mi(1-t)}
, (k+2)+(k+1)\mi +\me^{\frac{\pi}{2}\mi(1+t)},\dots, n +n\mi),
\]
and  $b$ could be realized as a representative path
\[
\alpha(t) = ((1+\mi)\me^{\frac{t}{p}},2+2\mi,\dots,n +n\mi).
\]
And $B^{orb}_n(\mathbb{C},\mathbb{Z}_p)$ admits several relations:
\begin{enumerate}
\item $b^p = e;$
\item $(bb_1)^p = (b_1b)^p$ ($p$ even);
\item $b_kb= bb_k (k>1)$;
\item $b_kb_{k+1}b_k = b_{k+1}b_k b_{k+1}$;
\item $b_k b_l = b_l b_k$ ($|k-l|>1$).
\end{enumerate}
$B^{orb}_n(\mathbb{C}^{\times},\mathbb{Z}_p)$ has just the same gererators but one less relation. $b^p$ is a non-trivial element in this group, since the first string cannot pass through the central point of the plane.

Then, we tried to employ Artin’s method, which computes the fundamental group of the complement of the closure of the ordinary braid, to the case of orbit braid. We have made generalizations of the first two steps of Artin:
\begin{itemize}
\item We established a faithful representation of $B^{orb}_n(\mathbb{C}^{\times},\mathbb{Z}_p)$ in the fundamental group of $\mathbf{D}^2\setminus Q_pn$, and by which we proved a representation of $B^{orb}_n(\mathbb{C},\mathbb{Z}_p)$ in a finite generated group is faithful.
\item Then we investigated some characteristics of G-homeomorphisms of $\mathbf{D}^2$  that fix boundary, to identify each such kind of G-homeomorphism with a braid element of $B^{orb}_n(\mathbb{C},\mathbb{Z}_p)$.

\end{itemize}

\section{Orbit braid action on a finite generated group}

\begin{definition}\label{def:repre}
 $R_{pn}$ is a group with generators $x_{ij}$, 
$0\leq i \leq p-1$ , $0\leq j \leq n-1$, and presentation given by
\begin{equation*}
\langle x_{00},x_{01},\cdots x_{p-1,n-1} | (\prod\limits_{k=i}^{i+p-1}x_{k \bmod p,0}) x_{ij}(\prod\limits_{k=i}^{i+p-1}x_{k \bmod p,0})^{-1}=x_{ij}, 0\leq i \leq p-1 ,  0< j \leq n-1 \rangle
\end{equation*}
$\rho_R$ is a homomorphism from 
$B^{orb}_n(\mathbb{C},\mathbb{Z}_p)$
to $Aut R_{pn}$, such that
$$\rho_R(b):\begin{cases}
x_{i0}\mapsto x_{(i+1)\bmod p,0}, &0\leq i \leq p-1\\
x_{ij}\mapsto x_{i0}^{-1}x_{ij}x_{i0}, &0\leq i \leq p-1 ,  j\neq0\\
\end{cases}$$
$$\rho_R(b_k):\begin{cases}
x_{ik}\mapsto x_{i,k+1}, &0\leq i \leq p-1\\
x_{i,k+1}\mapsto x_{i,k+1}x_{ik}x_{i,k+1}^{-1}, &0\leq i \leq p-1\\
x_{ij}\mapsto x_{ij}, &0\leq i \leq p-1,  j\neq k, k+1\\
\end{cases}$$

\end{definition}

\begin{theorem}\label{main}
Let 
\begin{equation*}
M_n^{orb} = \{Aut(\pi_1(\mathbf{D}^2  \setminus Q_{pn})) 
\cap Hom_G(\mathbf{D}^2  \setminus Q_{pn} , Id_{\partial \mathbf{D}^2} )
\}/<\rho_R(b^p)>,
\end{equation*}
where $Q_{pn}$ is the set of n points located at $(1,1), (2,2)\cdots(n,n)$ and their orbits under $\mathbb{Z}_p$ in the interior of 2-ball 
$\mathbf{D}^2$, and G-homeomorphisems are identity on the 
$\partial \mathbf{D}^2$, and $<\rho_R(b^p)>$ is generated by a  G-homeomorphisem which fixes outside $B_{\frac{5}{3}}(0,0) $ and is $2\pi$ rotation within $B_{\frac{4}{3}}(0,0) $. Then 
\begin{equation*}
M_n^{orb} \cong \rho_R(B^{orb}_n(\mathbb{C},\mathbb{Z}_p)) 
\cong B^{orb}_n(\mathbb{C},\mathbb{Z}_p).
\end{equation*}

\end{theorem}

In fact, $x_{ij}$ in Definition \autoref{def:repre} is a basis of
 $\pi_1(\mathbf{D}^2  \setminus Q_{pn})$  generated by a loop enclosing
$q_{ij} \in Q_{pn}$, with the base point $(0,0)$. 

\begin{lemma}
 $\rho_R$ is well defined.
\end{lemma}

\begin{proof}
Verify $\rho_R(b^p) = \rho_R(e)$, 
whenever $0\leq i,j \leq p-1 ,  j\neq0$ :
$$\rho_R(b^p):\begin{cases}
x_{i0}\stackrel{\rho_R(b)}{\mapsto} x_{(i+1)\bmod p,0} 
\stackrel{\rho_R(b)}{\mapsto} \cdots 
\stackrel{\rho_R(b)}{\mapsto} x_{(i+p)\bmod p,0} = x_{i0}\\
x_{ij}\stackrel{\rho_R(b)}{\mapsto} x_{i0}^{-1}x_{ij}x_{i0}
\stackrel{\rho_R(b)}{\mapsto} 
x_{(i+1)\bmod p,0}^{-1}( x_{i0}^{-1}x_{ij}x_{i0})x_{(i+1)\bmod p,0}
\stackrel{\rho_R(b)}{\mapsto}\cdots \\
\stackrel{\rho_R(b)}{\mapsto}
x_{(i+p-1)\bmod p,0}^{-1}\cdots
(x_{(i+p)\bmod p,0}^{-1}x_{(i+p)\bmod p,j}x_{(i+p)\bmod p,0})\cdots
x_{(i+p-1)\bmod p,0}\\
=(\prod\limits_{k=i}^{i+p-1}x_{k \bmod p,0})^{-1} x_{ij}(\prod\limits_{k=i}^{i+p-1}x_{k \bmod p,0})^{-(-1)} = x_{ij}\end{cases}$$

Verify $\rho_R((bb_0)^p) = \rho_R((b_0b)^p)$, 
whenever $0\leq i,j \leq p-1 ,  j\neq0,1$, $p$ even :
$$\rho_R((bb_0)^p):\begin{cases}
x_{i0}  \stackrel{\rho_R(b)}{\mapsto}  x_{(i+1)\bmod p,0}
\stackrel{\rho_R(b_0)}{\mapsto}   x_{(i+1)\bmod p,1}
\stackrel{\rho_R(b)}{\mapsto}
x_{(i+1)\bmod p,0}^{-1}x_{(i+1)\bmod p,1}x_{(i+1)\bmod p,0}\\
\stackrel{\rho_R(b_0)}{\mapsto}
x_{(i+1)\bmod p,1}^{-1}(x_{(i+1)\bmod p,1}x_{(i+1)\bmod p,0}x_{(i+1)\bmod p,1}^{-1})x_{(i+1)\bmod p,1}\\
=x_{(i+1)\bmod p,0}
\stackrel{\rho_R(b)}{\mapsto}\cdots\stackrel{\rho_R(b_0)}{\mapsto}
\cdots \cdots
\stackrel{\rho_R(b)}{\mapsto}\cdots\stackrel{\rho_R(b_0)}{\mapsto}
x_{(i+\frac{p}{2})\bmod p,0}  
;\\

x_{i1}\stackrel{\rho_R(b)}{\mapsto}x_{i0}^{-1}x_{i1}x_{i0}
\stackrel{\rho_R(b_0)}{\mapsto}x_{i0}
\stackrel{\rho_R(b)}{\mapsto}
\cdots \cdots
\stackrel{\rho_R(b_0)}{\mapsto}
x_{(i+\frac{p}{2})\bmod p,1}  
;\\

x_{ij}\stackrel{\rho_R(b)}{\mapsto}  x_{i0}^{-1} x_{ij} x_{i0}
\stackrel{\rho_R(b_0)}{\mapsto}   x_{i1}^{-1} x_{ij} x_{i1}
\stackrel{\rho_R(b)}{\mapsto}   
 (x_{i0}^{-1} x_{i1} x_{i0})^{-1} ( x_{i0}^{-1} x_{ij} x_{i0}) (x_{i0}^{-1} x_{i1} x_{i0})\\
= x_{i0}^{-1} (x_{i1}^{-1} x_{ij} x_{i1}) x_{i0}
\stackrel{\rho_R(b_0)}{\mapsto} 
x_{i1}^{-1} (x_{i1} x_{i0} x_{i1}^{-1})^{-1}  x_{ij} 
(x_{i1} x_{i0} x_{i1}^{-1}) x_{i1} =  x_{i0}^{-1} x_{i1}^{-1} x_{ij} x_{i1} x_{i0}\\ 
\stackrel{\rho_R(b)}{\mapsto}
\cdots \cdots
\stackrel{\rho_R(b_0)}{\mapsto}
(\prod\limits_{k=i}^{i+\frac{p}{2}-1}(x_{k(\bmod p),1}x_{k(\bmod p),0}))^{-1}x_{ij}
(\prod\limits_{k=i}^{i+\frac{p}{2}-1}(x_{k(\bmod p),1}x_{k(\bmod p),0}));\\ 

\end{cases}$$

$$\rho_R((b_0b)^p):\begin{cases}
x_{i0}  \stackrel{\rho_R(b_0)}{\mapsto}  x_{i1}
\stackrel{\rho_R(b)}{\mapsto}  x_{i0}^{-1} x_{i,1}x_{i0}
\stackrel{\rho_R(b_0)}{\mapsto}
x_{i0}
\stackrel{\rho_R(b)}{\mapsto}
x_{(i+1)\bmod p,0}\\
\stackrel{\rho_R(b_0)}{\mapsto}\cdots\stackrel{\rho_R(b)}{\mapsto}
\cdots \cdots
\stackrel{\rho_R(b_0)}{\mapsto}\cdots\stackrel{\rho_R(b)}{\mapsto}
x_{(i+\frac{p}{2})\bmod p,0} 
;\\

x_{i1}\stackrel{\rho_R(b_0)}{\mapsto}x_{i1}x_{i0}x_{i1}^{-1}
\stackrel{\rho_R(b)}{\mapsto}
(x_{i0}^{-1}x_{i1}x_{i0})x_{(i+1)\bmod p,0}(x_{i0}^{-1}x_{i1}x_{i0})^{-1}
\stackrel{\rho_R(b_0)}{\mapsto}
x_{i0}x_{(i+1)\bmod p,1}x_{i0}^{-1}\\
\stackrel{\rho_R(b)}{\mapsto}
x_{(i+1)\bmod p,1}
\stackrel{\rho_R(b_0)}{\mapsto}
\cdots \cdots
\stackrel{\rho_R(b)}{\mapsto}
x_{(i+\frac{p}{2})\bmod p,1}  
;\\

x_{ij}\stackrel{\rho_R(b_0)}{\mapsto}   x_{ij} 
\stackrel{\rho_R(b)}{\mapsto}   x_{i0}^{-1} x_{ij} x_{i0}
\stackrel{\rho_R(b_0)}{\mapsto}   
 x_{i1}^{-1} x_{ij} x_{i1}
\stackrel{\rho_R(b)}{\mapsto} 
x_{i0}^{-1} x_{i1}^{-1} x_{ij} x_{i1} x_{i0}\\ 
\stackrel{\rho_R(b_0)}{\mapsto}
\cdots \cdots
\stackrel{\rho_R(b)}{\mapsto}
(\prod\limits_{k=i}^{i+\frac{p}{2}-1}(x_{k(\bmod p),1}x_{k(\bmod p),0}))^{-1}x_{ij}
(\prod\limits_{k=i}^{i+\frac{p}{2}-1}(x_{k(\bmod p),1}x_{k(\bmod p),0}));\\ 
\end{cases}$$
$\rho_R(b_k b) = \rho_R(bb_k)\ (k > 0)$, 
$\rho_R(b_k b_{k+1} b_k) = \rho_R(b_{k+1} b_k b_{k+1})$, 
$\rho_R(b_k b_l) = \rho_R(b_l b_k)\ (\left| k-l \right|> 1)$
are easy to check.
\end{proof}

We first consider $B^{orb}_n(\mathbb{C}^{\times},\mathbb{Z}_p)$ which has the same generators as $B^{orb}_n(\mathbb{C},\mathbb{Z}_p)$. Let $F_{pn}$ be a free group with the same generators of $R_{pn}$, and $\rho_F:B^{orb}_n(\mathbb{C}^{\times},\mathbb{Z}_p) \rightarrow F_{pn}$ just as what $\rho_R$ does. Our aim is to prove that $\rho_F$ is faithful first and then $\rho_R$. 

But before this, we have to make some preparations. Let $F^m_{\mathbb{Z}_p}\mathbb{C}^{\times} \subset \mathbb{C}^{\times}$ be a set of fixed distinguished $m$ points and their orbits. Define $F^m_{\mathbb{Z}_p}(\mathbb{C}^{\times},n) = 
F_{\mathbb{Z}_p}(\mathbb{C}^{\times}\setminus  F^m_{\mathbb{Z}_p}\mathbb{C}^{\times},n)$.
 Let $\pi^{r}_{n}:
F^m_{\mathbb{Z}_p}(\mathbb{C}^{\times},n) \rightarrow 
F^m_{\mathbb{Z}_p}(\mathbb{C}^{\times},r)$ be the projection introduced from $(\mathbb{C}^{\times})^n = (\prod\limits^n_{i = r+1} \mathbb{C}^{\times})\times (\prod\limits^r_{i = 1} \mathbb{C}^{\times}) = (\mathbb{C}^{\times})^{n-r} \times (\mathbb{C}^{\times})^r \rightarrow (\mathbb{C}^{\times})^r$.
The projection $\pi^{n-1}_{n}$ is a fibration with fiber $F^{n-1}_{\mathbb{Z}_p}\mathbb{C}^{\times}$, whose proof is similar to \cite{FADELL1962}. 
Then we have 
\begin{lemma}
$\pi_2(F_{\mathbb{Z}_p}(\mathbb{C}^{\times},n-1)) =1$.
\end{lemma}

\begin{proof}
By the covering homotopy property of fibration we have an exact sequence
\[ 
\xymatrix{
 \pi_2(F^{n-1}_{\mathbb{Z}_p}\mathbb{C}^{\times}) \ar[r] &
\pi_2(F_{\mathbb{Z}_p}(\mathbb{C}^{\times},n)) \ar[r]^(0.46){\pi_1(\pi^{n-1}_{n})} & 
\pi_2(F_{\mathbb{Z}_p}(\mathbb{C}^{\times},n-1))  \\
}.
\]
$\pi_2(F^{n-1}_{\mathbb{Z}_p}\mathbb{C}^{\times}) $ is definitely trivial since $F^{n-1}_{\mathbb{Z}_p}\mathbb{C}^{\times}$ has the same homotopy type as wedge of $p(n-1)+1$ circles, whose universal cover is a infinite tree in $(p(n-1)+1)$-dimension space. The conclusion follows the induction of $n$.
\end{proof}
\begin{lemma}
 $\rho_F$ is faithful.
\end{lemma}

\begin{proof}
   \cite{Li2019} has proved that
\[ 1 \rightarrow P_n(\mathbb{C}^{\times},\mathbb{Z}_p) \rightarrow 
B^{orb}_n(\mathbb{C}^{\times},\mathbb{Z}_p) \rightarrow \mathbb{Z}_p^{\times n} \rtimes \Sigma_n \rightarrow 1
\]
is a short exact sequence. Define $Aut_{\mathbb{Z}_p}F_n \subset AutF_n $ such that  every $h \in Aut_{\mathbb{Z}_p}F_n$ satisfies
\[
c(h(x_{ij})) = h(x_{i+1(\bmod p), j})
\]
where $c \in Aut F_{pn} $ subjects to $c(x_{ij}) = x_{i+1(\bmod p), j}$. Now we have two short exact rows and they form a commutative diagram
\[
\xymatrix{
P_n(\mathbb{C}^{\times},\mathbb{Z}_p) \ar[r] \ar[d]^{\rho_F |_{ P_n(\mathbb{C}^{\times},\mathbb{Z}_p) }} &
B^{orb}_n(\mathbb{C}^{\times},\mathbb{Z}_p) \ar[r] \ar[d]^{\rho_F} & \mathbb{Z}_p^{\times n} \rtimes \Sigma_n \ar[d]^{\cong}\\
\text{ker}\rho \ar[r] & Aut_{\mathbb{Z}_p}F_n \ar[r] & Aut_{\mathbb{Z}_p}(F_n/[F_n,F_n])}
\]
By Five Lemma we just need to check that $\rho_F |_{ P_n(\mathbb{C}^{\times},\mathbb{Z}_p) }$ is monomorphism.

With the idea that the central point should be regarded as a fixed vertical string, We could easily see that the group      
  $P_n(\mathbb{C}^{\times},\mathbb{Z}_p)$ has generators in
  \autoref{fig:generators}
  \begin{gather*}
    \{A_{iqj}= \{(b_{i-1}\cdots b_0)b (b_0 \cdots b_{i-1} )\}^q b_i \cdots b_{j-1}   
    b_j^2 
    b_{j-1}^{-1} \cdots b_i^{-1} \{(b_{i-1}\cdots b_0) b(b_0 \cdots b_{i-1} )\}^{-q}: 
    \\0\leq i < j \leq n -1, 0 \leq q \leq p-1, \\ A_i = (b_{i-1}\cdots b_0 b b_0 \cdots b_{i-1})^p: 0\leq i \leq n -1\}
  \end{gather*}
\begin{figure}[hbt]
    \centering
    \includegraphics[scale = 0.3]{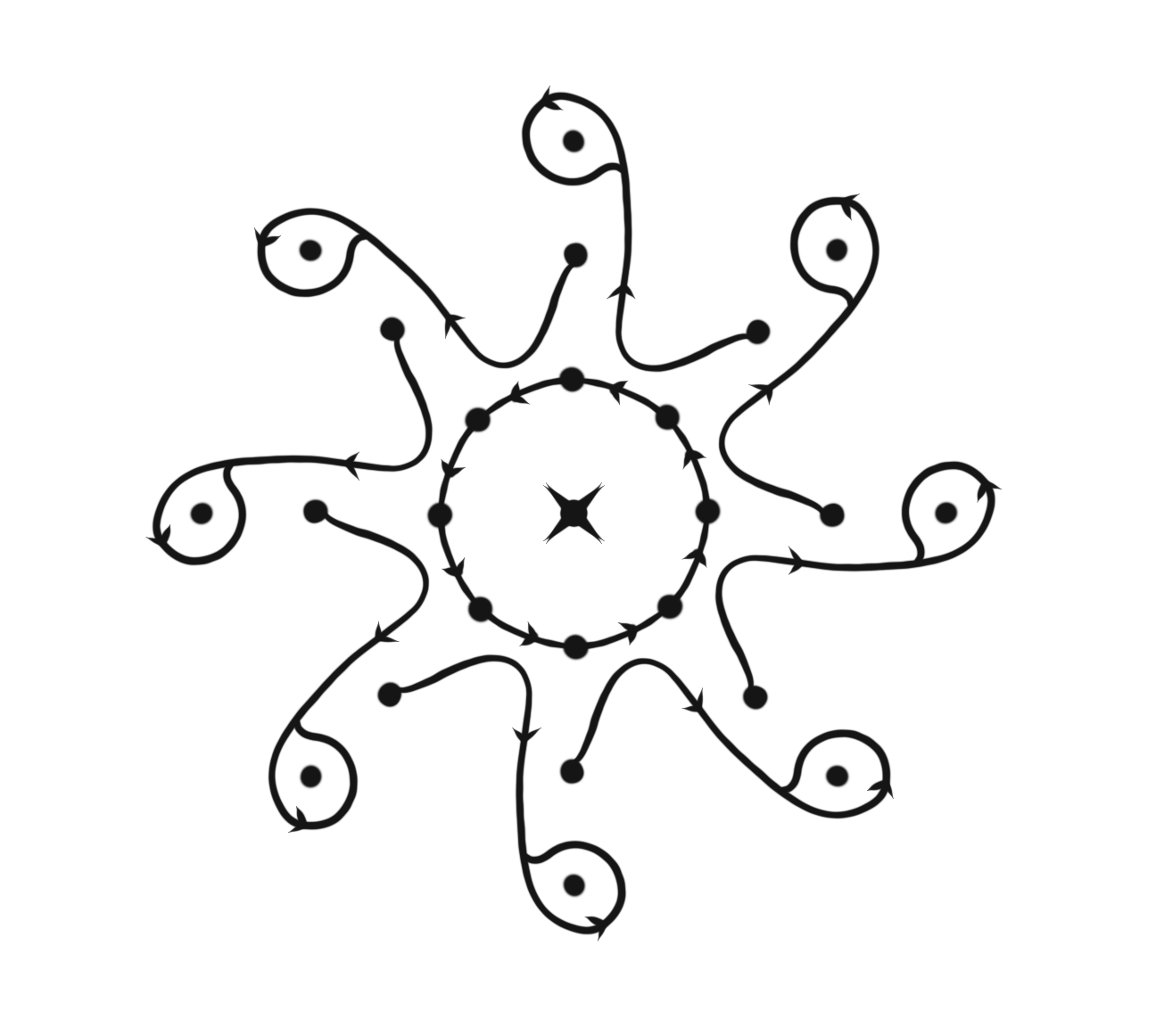}
    \caption{$A_{0,1,1}=b_0bb_0b_1^2(b_0bb_0)^{-1}$ and $A_0 = b^p$} 
    \label{fig:generators}
  \end{figure}

Note that $P_{n-1}(\mathbb{C}^{\times},\mathbb{Z}_p) = \{ A_{iqj}:0< i<j \leq n -1, 0 \leq q \leq p-1; A_i: i >0 \}$ is a subgroup of  $P_n(\mathbb{C}^{\times},\mathbb{Z}_p)$. Define a projection
\begin{align*}
\epsilon :P_{n}(\mathbb{C}^{\times},\mathbb{Z}_p) &\rightarrow 
P_{n-1}(\mathbb{C}^{\times},\mathbb{Z}_p)\\
A_{iqj}  &\mapsto 
\begin{cases} A_{iqj} & if \ 0< i < j \leq n -1 \\
1& if  \  1\leq j \leq n -1, i = 0 \end{cases}\\
A_{i}  &\mapsto 
\begin{cases} A_{i} & if \ 0< i \leq n -1 \\
1& if  \  i = 0 \end{cases}
\end{align*}
According to \cite{Li2019}, there is an isomorphism $l_n:P_n(\mathbb{C}^{\times},\mathbb{Z}_p) \rightarrow \pi_1(F_{\mathbb{Z}_p}(\mathbb{C}^{\times},n),\mathbf{x})$.

Now we have a commutative diagram
\[ 
\xymatrix{
1 \ar[r] \ar[d]^{id} &U_{n}(\mathbb{C}^{\times},\mathbb{Z}_p) = ker\epsilon 
\ar[r] \ar[d]^{l_n|_{ker \epsilon}} & P_{n}(\mathbb{C}^{\times},\mathbb{Z}_p)
\ar[r]^{\epsilon} \ar[d]^{l_n} &  P_{n-1}(\mathbb{C}^{\times},\mathbb{Z}_p)
\ar[r] \ar[d]^{l_{n-1}} & 1 \ar[d]^{id}\\
1 \ar[r] & \pi_1(F^{n-1}_{\mathbb{Z}_p}\mathbb{C}^{\times}) \ar[r] &
\pi_1(F_{\mathbb{Z}_p}(\mathbb{C}^{\times},n)) \ar[r]^(0.46){\pi_1(\pi^{n-1}_{n})} & 
\pi_1(F_{\mathbb{Z}_p}(\mathbb{C}^{\times},n-1)) \ar[r] & 1 \\
}
\]
where the rows are exact, and $\pi_1(F^{n-1}_{\mathbb{Z}_p}\mathbb{C}^{\times}) = \pi_1(\mathbb{C} \setminus Q_{p(n-1)+1})$. By Five Lemma, $ker \epsilon$ is isomorphic to $\pi_1(F^{n-1}_{\mathbb{Z}_p}\mathbb{C}^{\times})$.
This means that $U_{n}(\mathbb{C}^{\times},\mathbb{Z}_p)$ is a free group with a free basis $\{A_0,A_{0qj}:0< j \leq n -1, 0 \leq q \leq p-1\}$ since $\{l_n(A_{0},l_n(A_{0qj}:0< j \leq n -1, 0 \leq q \leq p-1\}$ is a free basis of  
$\pi_1(F^{n-1}_{\mathbb{Z}_p}\mathbb{C}^{\times})$.

Intuitively, we could examine what $A_{i(qj)}A_{0rk}A_{i(qj)}^{-1}$ is, to see why $ker\epsilon$ is generated by $\{A_0,A_{0qj}:0< j \leq n -1, 0 \leq q \leq p-1\}$. Just shift the loop of $A_{0rk}$ down through the loop of $A_{i(qj)}$, during which the initial point and end point of the first loop is maintained so that it can always been presented as composition of $A_{0ab}$ and $A_0$. Than we just reduce the loops of $A_{i(qj)}$ and $A_{i(qj)}^{-1}$ to identity. For example, $A_{1}A_{001}
= (A_0A_{001}A_0^{-1})A_{1}$. Other cases could be sorted and handled by induction. As a result, every element $a \in P_{n}(\mathbb{C}^{\times},\mathbb{Z}_p)$ can be represented uniquely as the normal form 
\[
a = a_2 \cdots a_n
\]
with $a_j \in U_{j}(\mathbb{C}^{\times},\mathbb{Z}_p)$.

We have seen that $P_{n}(\mathbb{C}^{\times},\mathbb{Z}_p)$ acts by conjugation as a group of automorphisms of $U_{n+1}(\mathbb{C}^{\times},\mathbb{Z}_p)$. The automorphisms of fundamental groups could be regarded as the transformations of the loops of $U_{n+1}(\mathbb{C}^{\times},\mathbb{Z}_p)$, as the same fashion discussed above. Thus if we define $f: U_{n+1}(\mathbb{C}^{\times},\mathbb{Z}_p) \rightarrow F_n$ by $f(A_{0qj}) = x_{qj}(0\leq q \leq p-1 , 0< j \leq n-1),f(A_0) = 0$, the conjugate action would make the following diagram commutative:
\[
\xymatrix{
P_{n}(\mathbb{C}^{\times},\mathbb{Z}_p) \ar[r]^(0.4){c} \ar[d]^{id} &
AutU_{n+1}(\mathbb{C}^{\times},\mathbb{Z}_p) \ar[d]^{Aut(f)} \\
P_{n}(\mathbb{C}^{\times},\mathbb{Z}_p)\ar[r]^{\rho_F|_{P_{n}(\mathbb{C}^{\times},\mathbb{Z}_p)}} &
Aut_{\mathbb{Z}_p}F_n }.
\]
Because the image of $c$ leaves $A_0$ constant, we conclude that ker$\rho_F|_{P_{n}(\mathbb{C}^{\times},\mathbb{Z}_p)}$ is the subgroup of elements in $P_{n}(\mathbb{C}^{\times},\mathbb{Z}_p)$ which commute with $U_{n+1}(\mathbb{C}^{\times},\mathbb{Z}_p)$.

Assume that there is an element $a \in \text{ker} \rho_F|_{P_{n}(\mathbb{C}^{\times},\mathbb{Z}_p)}$ such that $a = a_2 \cdots a_i$. Then $a_i$ commute with $A_{00i}$. Let $\pi = b_0b_1\cdots b_{i-1}$. We have
\begin{align*}
\pi A_{00i} \pi^{-1} &= A_{001}\\
a_i A_{00i} a_i^{-1} &= A_{00i}\\
(\pi a_i \pi^{-1})A_{001}(\pi a_i \pi^{-1})^{-1} &= (\pi a_i \pi^{-1})(\pi A_{00i} \pi^{-1} )(\pi a_i \pi^{-1})^{-1}\\
&= \pi a_i A_{00i} a_i^{-1} \pi^{-1} \\
&= \pi A_{00i}\pi^{-1}\\
&= A_{001}
\end{align*}
However, $\pi a_i \pi^{-1}$ is in $U_{n+1}(\mathbb{C}^{\times},\mathbb{Z}_p)$. It must be $A_{001}^l$ since it commutes with $A_{001}$. Then $a_i = \pi^{-1} A_{001}^l\pi = A_{00i}^l$, which is a contradiction.

Now we consider $\rho_R$.
\begin{theorem}\label{thm:faithful}
 $\rho_R$ is faithful.
\end{theorem}
 If $\beta \in B^{orb}_n(\mathbb{C},\mathbb{Z}_p)$ and $\rho_R(\beta) = id$, and $\beta'$ is the corresponding element in $B^{orb}_n(\mathbb{C}^{\times},\mathbb{Z}_p)$, then
$\rho_F(\beta')(x_{ij}) = A_{ij} x_{\mu_{ij}} A_{ij}^{-1} $ must be $x_{ij}$ in $R_{pn}$. 
Because when the relation of $R_{pn}$ is reduced, the odevity of words is kept, $\mu_{ij}$ must be $ij$. If $ j > 0$, $A_{ij}$ must be multiple of $(\prod\limits_{k=i}^{i+p-1}x_{k \bmod p,0})$ since other words don't commute with $x_{ij}$. If $j = 0$, since the words on the right and the left of $x_{ij}$ are symmetric, $x_{ij}$ cannot commute with other words as a part of $(\prod\limits_{k=i}^{i+p-1}x_{k \bmod p,0})$. Thus $A_{i0}$ must be $1$.

Note that 
\begin{align*}
\rho_F(b) (\prod\limits_{i=0}^{p-1}(\prod\limits_{j=n-1}^{0} x_{ij})) &= 
(x_{00})^{-1} \prod\limits_{i=0}^{p-1}(\prod\limits_{j=n-1}^{0} x_{ij}) x_{00}\\
\rho_F(b_k) (\prod\limits_{i=0}^{p-1}(\prod\limits_{j=n-1}^{0} x_{ij}))&= 
\prod\limits_{i=0}^{p-1}(\prod\limits_{j=n-1}^{0} x_{ij})  .
\end{align*}
Thus
\[
\rho_F(\beta') (\prod\limits_{i=0}^{p-1}(\prod\limits_{j=n-1}^{0} x_{ij})) = 
A \prod\limits_{i=0}^{p-1}(\prod\limits_{j=n-1}^{0} x_{ij}) A^{-1} .
\]
Thus $A_{ij}$ would contain the same number of $(\prod\limits_{k=i}^{i+p-1}x_{k \bmod p,0})$. Suppose $A_{ij} = (\prod\limits_{k=i}^{i+p-1}x_{k \bmod p,0})^m$. In conclusion, $\rho_F(\beta')$ would act just as $\rho_F((b^p)^m)$. Because we have proved $\rho_F$ is faithful, $\beta'$ thus $\beta$ must be $(b^p)^m$. So $\beta$ is identity in $B^{orb}_n(\mathbb{C},\mathbb{Z}_p)$.
\end{proof}

\section{characterizations of orbit braid representation}
\begin{lemma}
 If $ \rho \in Aut R_{pn} \setminus \{ id_{R_{pn}} \}$:  (1)could be described as a simplest form:
\begin{equation}\label{equ:standard}
\rho (x_{ij}) = A_{ij} x_{\mu_{ij}} A_{ij}^{-1}   
\end{equation}
where $(i,j) \mapsto \mu_{ij}$ is a permutation, and (2)
\begin{equation}\label{equ:conjugation}
\rho (\prod\limits_{i=0}^{p-1}(\prod\limits_{j=n-1}^{0} x_{ij})) = 
A \prod\limits_{i=0}^{p-1}(\prod\limits_{j=n-1}^{0} x_{ij}) A^{-1}  
\end{equation}
and (3)
\begin{equation}\label{equ:circle}
c(A_{ij}) = A_{i+1(\bmod p), j}  
\end{equation}
\begin{equation}\label{equ:circle2}
c(x_{\mu_{ij}}) = x_{\mu_{i+1(\bmod p), j}}  
\end{equation}
where $c \in Aut R_{pn} $ subjects to $c(x_{ij}) = x_{i+1(\bmod p), j}$.
Then either (a) there is a pair $(i,j)$ such that $x_{\mu_{ij}}A_{ij}^{-1}$
could be absorbed by $A_{f(ij)}$; or (b) there is a pair $(i,j)$ 
such that  $A_{ij}^{-1}$ could absorb $A_{f(ij)}x_{\mu_{f(ij)}}$, 
where ${f(ij)}$ is the pair following $(i,j)$ in the order of (2.2.1).
\label{lem:absorb}
\end{lemma}

\begin{remark}\label{rem:action}
These three conditions are from the geometrical perspective. Every homeomorphisms on the disk would introduce the homomorphisms of fundamental group with the form of condition(1).  

$\prod\limits_{i=0}^{p-1}(\prod\limits_{j=n-1}^{0} x_{ij})$ is a loop starts at 
$(0,0)$ and encloses all the points and then return to $(0,0)$. Every
$\rho_R(b_k)$ would leave this loop fixed and $\rho_R(b) $ would make it be a conjugation. Thus we have condition(2).

Every G-homomorphisms under $\mathbb{Z}_p$-action would lead to condition(3).
\end{remark}

\begin{proof}
We have 
$$\prod\limits_{i=0}^{p-1}(\prod\limits_{j=n-1}^{0}  A_{ij} x_{\mu_{ij}} A_{ij}^{-1})  =  A \prod\limits_{i=0}^{p-1}(\prod\limits_{j=n-1}^{0} x_{ij}) A^{-1}$$
Thus
$$\prod\limits_{i=0}^{p-1}(\prod\limits_{j=n-1}^{0}  A ^{-1} A_{ij}  x_{\mu_{ij}} A_{ij}^{-1}A)  =  \prod\limits_{i=0}^{p-1}(\prod\limits_{j=n-1}^{0} x_{ij}) = x_{0, n-1} x_{0, n-2} \cdots x_{p-1 , 0}$$

We move $x_{\mu_{0,n-1}}$ from the right to the left:
\begin{equation*}\label{equ:move}
 x_{0, n-1}^{-1} ( A ^{-1} A_{0, n-1}  x_{\mu_{0, n-1}} A_{0, n-1}^{-1}A) \cdots
 (A ^{-1} A_{p-1,0}  x_{\mu_{p-1,0}} A_{p-1,0}^{-1}A) =
 x_{0, n-2} \cdots x_{p-1 , 0}  \tag{$\star$}
\end{equation*}
The $x_{0, n-1}^{-1}$ must be absorbed. Since there is equal amount(could be negtive) of $x_{0, n-1}$  before and after each $x_{ij}$ in each parenthesis, the $x_{0, n-1}^{-1}$ must be "transferred" alone these parentheses and eventually absorbed by a $x_{\mu_{ij}}$ that is $x_{0, n-1}$.

Suppose $x_{\mu_{0,n-1}}$ is $x_{0, n-1}$, then the transfer progress should be ended in the first parenthesis, because if $ A ^{-1} A_{0, n-1}$ has a multiple of $x_{0, n-1}$, then so is  $A_{0, n-1}^{-1}A$, which elements in subsequent parenthesis could not absorb. Thus  $x_{0, n-1}^{-1}$ would commute with $ A ^{-1} A_{0, n-1}$, and is absorbed by $x_{\mu_{0,n-1}}$.Therefore  $ A ^{-1}$ is multiple of $(\prod\limits_{k=0}^{0+p-1}x_{k \bmod p,0})$. Now we move the $x_{0, n-2}$ from the right to the left. Eventually we would be in two cases:

(i)The process depicted above can be done till the end. Every $A ^{-1} A_{ij}  x_{\mu_{ij}} A_{ij}^{-1}A$ is equal to $x_{ij}$. Then $x_{\mu_{ij}} $ is $x_{ij}$, and each $A ^{-1} A_{ij}$ could be reduce to identity. By \eqref{equ:circle} , $A $ and $A_{ij}$ are both identity. Thus $\rho$ is identity, which contradicts what we'd assumed.

(ii) If it's not case(i), we can always assume that $x_{\mu_{0,n-2}}$ is $x_{0, n-1}$ in \eqref{equ:move}. Since $x_{\mu_{0,n-1}}$ is not $x_{0, n-1}$, we have two ways to reduce  $x_{0, n-1}$ : 
either $x_{\mu_{0,n-1}} A_{0,n-1}^{-1}A $ is absorbed by $A ^{-1} A_{0, n-2} $, which implies (a) is true; 
or there is exactly one $x_{0, n-1}$ in $ A ^{-1} A_{0, n-1}$ which can absorb $x_{0, n-1}^{-1}$, such that $x_{0, n-1}^{-1}$ could "jump over" $x_{\mu_{0,n-1}}$, then reduce $x_{\mu_{0,n-2}}$. This means (b) is true.\end{proof}

\begin{lemma}\label{lem:length}
Under the conditions and notations of Lemma 2, if $\rho$ satisfies Lemma 2(a), 
\[ l(\rho \circ \rho_R(b_{j-1})) < l( \rho ) \ \text{or} \ 
 l(\rho \circ \rho_R(b^{-1})) < l( \rho ) ,\]
and if $\rho$ satisfies Lemma 2(b), 
\[ l(\rho \circ \rho_R(b_{j-1}^{-1})) < l( \rho ) \ \text{or} \ 
 l(\rho \circ \rho_R(b)) < l( \rho ) ,\]
where $l$ is length function, i.e., $l(\rho) = \sum\limits_{i,j} (\text{minimum letter lengths of the words } A_{ij} x_{\mu_{ij}} A_{ij}^{-1})$.

\end{lemma}

\begin{proof}
Case(a-i): There is a $j$ such that $x_{\mu_{ij}}A_{ij}^{-1}$
could be absorbed by $A_{i,j-1}$, which means $A_{i,j-1} = A_{ij}x_{\mu_{ij}}^{-1}B_{i,j-1} $ for any $0 \leq i \leq p-1$ by Lemma 2 Condition(3). We have
\begin{align*}
x_{ij}&\stackrel{\rho}{\mapsto} A_{ij}  x_{\mu_{ij}} A_{ij}^{-1}\\
x_{i, j-1}&\stackrel{\rho}{\mapsto} (A_{ij}x_{\mu_{ij}}^{-1}B_{i,j-1})x_{\mu_{i,j-1}}(A_{ij}x_{\mu_{ij}}^{-1}B_{i,j-1})^{-1}\\
x_{ij}&\stackrel{\rho_R({b_{j-1}})}{\mapsto} x_{ij} x_{i,j-1} x_{ij}^{-1}\\ &\stackrel{\rho}{\mapsto} (A_{ij}  x_{\mu_{ij}} A_{ij}^{-1})
(A_{ij}x_{\mu_{ij}}^{-1}B_{i,j-1})x_{\mu_{i,j-1}}(A_{ij}x_{\mu_{ij}}^{-1}B_{i,j-1})^{-1}  (A_{ij}  x_{\mu_{ij}} A_{ij}^{-1})^{-1} \\
&=A_{ij}B_{i,j-1} x_{\mu_{i,j-1}}B_{i,j-1}^{-1}A_{ij}^{-1}\\
x_{i,j-1}&\stackrel{\rho_R({b_{j-1}})}{\mapsto}x_{ij}\\
&\stackrel{\rho}{\mapsto} A_{ij}  x_{\mu_{ij}} A_{ij}^{-1}
\end{align*}
for any $0 \leq i \leq p-1$. Thus by comparison we get $l(\rho \circ \rho_R(b_{j-1})) < l( \rho ) $.

Case(a-ii): $x_{\mu_{i0}}A_{i0}^{-1}$
could be absorbed by $A_{i + 1,n-1}$, which means $A_{i+1,n-1} = A_{i0}x_{\mu_{i0}}^{-1}B_{i+1,n-1} $ for any $0 \leq i \leq p-2$ by Lemma 2 Condition(3). We have
\begin{align*}
x_{i0}&\stackrel{\rho}{\mapsto} A_{i0}  x_{\mu_{i0}} A_{i0}^{-1}\\
x_{i +1, n-1}&\stackrel{\rho}{\mapsto} (A_{i0}x_{\mu_{i0}}^{-1}
B_{i +1, n-1})x_{\mu_{i +1, n-1}}(A_{i0}x_{\mu_{i0}}^{-1}B_{i +1, n-1})^{-1}\\
x_{i0}&\stackrel{\rho_R(b^{-1})}{\mapsto}x_{i+1,0}\\
&\stackrel{\rho}{\mapsto} A_{i+1,0}  x_{\mu_{i+1,0}} A_{i+1,0}^{-1}\\
x_{i +1, n-1}&\stackrel{\rho_R(b^{-1})}{\mapsto} x_{i0} x_{i +1, n-1} x_{i0}^{-1}\\ &\stackrel{\rho}{\mapsto} (A_{i0}  x_{\mu_{i0}} A_{i0}^{-1})
(A_{i0}x_{\mu_{i0}}^{-1}B_{i +1, n-1})x_{\mu_{i +1, n-1}}(A_{i0}x_{\mu_{i0}}^{-1}B_{i +1, n-1})^{-1}  (A_{i0}  x_{\mu_{i0}} A_{i0}^{-1})^{-1} \\
&=A_{i0}B_{i +1, n-1} x_{\mu_{i +1, n-1}}B_{i +1, n-1}^{-1}A_{i0}^{-1}
\end{align*}
Since  \eqref{equ:circle} , letter length of  $A_{i0}  x_{\mu_{i0}} A_{i0}^{-1}$ is equal to 
$A_{i+1,0}  x_{\mu_{i+1,0}} A_{i+1,0}^{-1}$. Thus we have 
$l(\rho \circ \rho_R(b^{-1})) < l( \rho ) $.

Case(b) is similar to the discussion above.
\end{proof}

\begin{theorem}
$\rho \in AutR_{pn}$ subjects to \eqref{equ:standard}, \eqref{equ:conjugation}, \eqref{equ:circle} and \eqref{equ:circle2} if and only if
$\exists\  \sigma \in B^{orb}_n(\mathbb{C},\mathbb{Z}_p)$ such that
$\rho $ is an element of the equivalent class $\rho_R(\sigma)$.

\end{theorem}

\begin{proof}
"If": We have discussed in Remark \autoref{rem:action} .

"Only if": If $l(\rho)=pn$, then $\rho = id_{R_{pn}} \in \rho_R(e)$. Otherwise by Lemma \ref{lem:absorb} and Lemma \autoref{lem:length} we have a $\sigma \in B^{orb}_n(\mathbb{C},\mathbb{Z}_p)$ such that $l(\rho \circ \rho_R(\sigma)) < l(\rho)$. By induction we can render $ \rho$ as a product of elements from $Im(\rho_R)$. Thus the conclusion follows.
\end{proof}

\noindent Proof of Theorem \autoref{main} : Due to the discussion of Remark \autoref{rem:action} , we have a map
$M_n^{orb} \rightarrow \rho_R(B^{orb}_n(\mathbb{C},\mathbb{Z}_p)) $. Obviously it is a well defined homomorphism. It is a monomorphism exactly as the proof of Theorem \autoref{thm:faithful} .

On the other hand, we can realize $\rho_R(b)$ by a 
$\mathbb{Z}_p$-homeomorphism which fixes outside $B_{\frac{5}{3}}(0,0) $ and is $\frac{2\pi}{p}$ rotation within $B_{\frac{4}{3}}(0,0) $. And we realize 
$\rho _F(b_k)$ by a $\mathbb{Z}_p$-homeomorphism which interchanges 
$q_{ik}$ and $q_{i,k+1}$ and fixes outside those little disks which include them.
Therefore $\rho_R(B^{orb}_n(\mathbb{C},\mathbb{Z}_p)) \subset M_n^{orb}$.

\bibliographystyle{alpha}
\bibliography{./reference/braids_links}
\end{document}